\documentclass[11pt,reqno]{amsart}
\usepackage[margin=1.1in]{geometry}
\usepackage[T1]{fontenc}\usepackage[utf8]{inputenc}
\usepackage{lmodern,amsmath,amssymb,booktabs,array,microtype,needspace}
\usepackage[hidelinks]{hyperref}
\hypersetup{pdfauthor={Enkai Zhang},pdftitle={The additive-square-free spectrum below 139 for four-element alphabets}}
\newtheorem{theorem}{Theorem}
\newtheorem{proposition}[theorem]{Proposition}
\newtheorem{lemma}[theorem]{Lemma}
\theoremstyle{remark}\newtheorem{remark}[theorem]{Remark}
\newcommand{\HH}{\mathcal H}
\newcommand{\Z}{\mathbb Z}
\title[The additive-square-free spectrum below 139]{The additive-square-free spectrum below 139 for four-element alphabets}
\author{Enkai Zhang}
\address{University of Toronto Scarborough, Toronto, Ontario, Canada}
\email{ek.zhang@mail.utoronto.ca}
\subjclass[2020]{68R15, 11B75, 05D10}
\keywords{additive square, abelian square, four-letter alphabet, exhaustive search, real parameters}
\date{}
\providecommand{\wordentry}[2]{\par\Needspace{6\baselineskip}\noindent\textbf{#1}\par\smallskip\noindent{\small\ttfamily #2}\par\medskip}
\providecommand{\planeentry}[3]{\par\Needspace{3\baselineskip}\noindent\makebox[2.4em][l]{\small$W_{#1}$}{\small\ttfamily#2}\\\hspace*{2.4em}{\small\ttfamily#3}\par\vspace{1pt}}
\providecommand{\MainPrefix}{}
\providecommand{\SuppPrefix}{}
\hypersetup{hypertexnames=false}
\usepackage{amsthm}
\begin{document}
\hypertarget{article-start}{}
\renewcommand{\SuppPrefix}{Supplement }

\begin{abstract}
For a finite real alphabet $T$, let $g(T)$ be the supremum of the lengths of
words over $T$ containing no two adjacent nonempty blocks of equal length and equal
sum. We classify all four-element alphabets with $g(T)<139$. Up to affine
equivalence, these are the known endpoint-sum family and seven isolated
classes. Exhaustive finite searches give the isolated values, while 77 explicit
words prove that every remaining real alphabet admits a word of length 139.
The latter argument uses exact intersections of finite unions of lines, so it
covers irrational parameters as well. The cutoff is sharp because
$g(\{0,2,3,7\})=139$. We also obtain a sharp cutoff of 167 within the family
$\{0,1,x,x+2\}$, $x>1$. Every four-element integer alphabet of primitive
diameter at most seven has finite $g$.
\end{abstract}
\maketitle

\section{The classification}
An \emph{alphabet} is a finite set of real numbers, and a \emph{word} is a
finite sequence of its elements. A \emph{factor} is a contiguous block.
Write $|W|$ for the length of a word and $\sum W$ for the sum of its letters.
An \emph{additive square} is a factor $UV$ such that
$|U|=|V|>0$ and $\sum U=\sum V$. For example, $03\,12$ is an additive square,
since its two length-two blocks both sum to three. The blocks need not be
identical or have the same letter multiplicities.

Put
\[
 g(T)=\sup\{|W|:W\text{ contains no additive square and is a word over }T\},
\]
allowing $g(T)=\infty$. The \emph{spectrum} is the set of values $g(T)$
as $T$ ranges over all four-element real alphabets. Thus $g(T)=n<\infty$ means that some word of length
$n$ avoids additive squares, but every word of length $n+1$ contains one.
A witness of length $n$ proves only the lower bound; the upper bound requires
excluding every word of length $n+1$.

Write $T\sim T'$ if $T'=\{\mu t+\nu:t\in T\}$ for real $\mu\ne0$ and $\nu$.
This \emph{affine equivalence} includes reflection and preserves $g$:
translation contributes equally to two blocks of equal length, and nonzero
scaling preserves equality of their sums. Define the endpoint-sum family
\[
 \HH=\bigl\{\{a,b,c,d\}:a<b<c<d,\ a+d=b+c\bigr\}.
\]
For an integer alphabet beginning at zero, \emph{primitive} means that the
nonzero letters have greatest common divisor one. Its diameter is its largest
letter. These conventions remove redundant integer rescalings.

\Needspace{18\baselineskip}
\begin{theorem}\label{M-thm:main}
For every four-element real alphabet $T$, the complete classification below
139 is as follows.
\begin{center}
\begin{tabular}{@{}ll@{}}
\toprule
Condition & $g(T)$\\\midrule
$T\sim\{0,1,2,3\}$&$50$\\
$T\sim\{0,1,3,4\},\{0,1,4,5\},\{0,2,3,5\}$&$55$\\
$T\sim\{0,3,4,7\}$&$58$\\
$T\in\HH$, excluding the preceding classes&$60$\\
$T\sim\{0,1,2,4\}$&$62$\\
$T\sim\{0,1,2,5\}$&$86$\\
$T\sim\{0,1,3,5\}$&$88$\\
$T\sim\{0,1,3,6\}$&$110$\\
$T\sim\{0,1,2,6\}$&$111$\\
$T\sim\{0,2,3,6\}$&$113$\\
$T\sim\{0,1,4,6\}$&$118$\\
Every other four-element real alphabet&$\ge139$\\\bottomrule
\end{tabular}
\end{center}
The last bound is attained: $g(\{0,2,3,7\})=139$.
\end{theorem}

The first four rows are the classification of Freedman and Brown
\cite[Theorem 1]{M-FB16}. The value $62$ had appeared in two machine-assisted computational records.\footnote{TheoremDB,
\href{https://theoremdb.org/statements/P2824/}{P2824} (July 2026), and
J.~Huhta, Word Structures Lab, \href{https://github.com/word-structures/combinatorics-on-words-research/blob/186df9debe7f799c367be65f921dcfe1f97ce2f2/MATH_CLAIMS.md}{\texttt{MATH\_CLAIMS.md}, row 54} (31 August 2026).
The complete search used here is described in Supplement S5.} Outside $\HH$, the seven isolated classes are exactly the normalized
primitive alphabets $\{0,a,b,c\}$ with $c\le6$, up to reflection.
Section~\ref{M-sec:finite} gives their finite-search results. Sections 3--5 prove that
there are no other alphabets below the cutoff: 77 fixed words cover the
complement of the listed classes throughout the real parameter plane.
Theorem~\ref{M-thm:line} treats a particular one-parameter family below 167;
it is not a classification of the whole plane at that larger cutoff.

\paragraph{Previous work.}
An \emph{abelian square} consists of adjacent blocks with equal multiplicities
of every letter. It is an additive square under every numerical assignment.
Justin \cite{M-Jus72} considered the problem in its semigroup form.
Brown and Freedman \cite{M-BF87} conjectured that every infinite bounded
sequence of positive integers contains consecutive blocks of the same positive
length and sum. Later formulations appear in \cite{M-PV94,M-HH00}.
In a recent preprint Shallit \cite{M-Sh26} gives a counterexample on six
positive integers. The minimum size of an integer alphabet admitting an
infinite additive-square-free word is consequently four, five or six.
Au, Robertson and Shallit \cite[Sections~4--5]{M-ARS11} studied the version
with block sums taken modulo an integer, including computations of extremal
lengths. Grytczuk \cite{M-Gry08} surveys related repetition problems for
color sequences along paths in graphs and lines in the plane.
For additive complexity and approximate repetitions see \cite{M-ABJS12,M-Bro12}.

An additive cube consists of three consecutive blocks of equal length and sum.
Following \cite{M-CCSS14,M-Rao15}, Lietard and Rosenfeld \cite{M-LR20}
proved cube avoidability for every four-element alphabet outside the affine
class of $\{0,1,2,3\}$. A recent preprint of Jamet and Rosenfeld
\cite[Theorem~4]{M-JR26} gives an infinite word avoiding additive cubes
on $\{0,1,2,3\}$.
Infinite square avoidance is possible on suitable
real alphabets via abelian-square avoidance \cite{M-Ker92,M-FB16}, and on a vector
alphabet in $\Z^2$ \cite{M-RR18}.
Freedman and Brown \cite{M-FB16} initiated the finite extremal problem, classified
$\HH$, and showed $g\ge51$ elsewhere. Further context appears in
\cite{M-AM25,M-V26}. The classification below concerns four real letters and finite values of $g$.

\section{Finite extremal values}\label{M-sec:finite}
We consider $\{0,a,b,c\}$ with $0<a<b<c$ and $\gcd(a,b,c)=1$.
Reflection $x\mapsto c-x$ is the only remaining nontrivial affine equivalence
between normalized primitive integer sets. In particular, $c$ is invariant
within such a class.

\Needspace{13\baselineskip}
\begin{proposition}\label{M-prop:layers}\label{M-prop:small}\label{M-prop:three}
Outside $\HH$, the normalized primitive classes through diameter seven have
the following values, up to reflection.
\begin{center}\small
\begin{tabular}{@{}cl@{}}\toprule
$c$&Alphabets and values of $g$\\\midrule
$4$&$\{0,1,2,4\}:62$\\
$5$&$\{0,1,2,5\}:86;\quad\{0,1,3,5\}:88$\\
$6$&$\{0,1,3,6\}:110;\quad\{0,1,2,6\}:111$\\
 &$\{0,2,3,6\}:113;\quad\{0,1,4,6\}:118$\\
$7$&$\{0,2,3,7\}:139;\quad\{0,1,3,7\}:141;\quad\{0,2,4,7\}:145$\\
 &$\{0,1,2,7\}:148;\quad\{0,1,4,7\}:161;\quad\{0,1,5,7\}:167$\\\bottomrule
\end{tabular}\end{center}
\end{proposition}

\begin{proof}
The search begins with the empty word and tries each possible next letter.
If an extension of a word avoiding additive squares creates one, the new
occurrence must end at the appended letter. With prefix sums $S_i$, the last
two blocks of length $k$ have equal sum exactly when
\begin{equation}
 S_n-2S_{n-k}+S_{n-2k}=0.\label{M-eq:suffix}
\end{equation}
Reject an extension if this holds for some $1\le k\le\lfloor n/2\rfloor$;
otherwise continue recursively. Induction on length shows that this procedure
visits every word avoiding additive squares, and no others.

For each alphabet, the completed search has a surviving word at the
listed length $g$ and no survivor at length $g+1$. The witnesses in
Supplement S1 prove the lower bounds; the empty final layers prove
the upper bounds, since every longer avoiding word would have an
avoiding prefix in that layer. Supplement S5 specifies the complete
runs and their subtree partitions.

The author developed the search and verification programs in the
accompanying archive with AI assistance. For $\{0,1,2,4\}$ the last nonempty layer contains exactly two words,
one the reversal of the other. For the values 86 and 88,
the archive also includes programs that exhaust all 3372 subtrees rooted at
avoiding prefixes of length eight. Their totals agree with the single-process
search. The search invariant and the completion of every subtree, rather than
agreement of totals alone, establish exhaustiveness.

Every $c=6,7$
search was completed by a graphics-processor breadth-first implementation
and at least one CPU depth-first program using
\eqref{M-eq:suffix}. Supplement S5 gives the number of distinct prefixes in each completed
tree, including the empty word. Supplement S1 prints all thirteen
maximum-length witnesses. This proves every entry for the stated
normalized classes through diameter seven.
\end{proof}

\paragraph{Integer alphabets of small diameter.}
For an integer alphabet, subtract its least element and divide by the gcd of
the resulting differences. The greatest element of this normalized alphabet is
its primitive diameter. At diameter seven there are $\binom62=15$
choices of $1\le a<b\le6$ in $\{0,a,b,7\}$. All are primitive because
$7$ is prime. Reflection sends $(a,b)$ to $(7-b,7-a)$; its three fixed
pairs $(1,6),(2,5),(3,4)$ are precisely the endpoint-sum alphabets.
The remaining twelve choices form six pairs, one representative of each
in the $c=7$ row of Proposition~\ref{M-prop:layers}.

If the primitive diameter is at most seven, then $g(T)\le167$:
the endpoint-sum family has $g(T)\le60$ by \cite{M-FB16}, and
Proposition~\ref{M-prop:layers} exhausts the remaining classes. The bound
is attained at $\{0,1,5,7\}$. In particular, every word of length 168 on
such an alphabet contains an additive square. The values $141,145,148,161$
in the last row are attained in the full spectrum. The classification on
the whole real parameter plane ends at $g<139$.

\section{Normalization and failure sets}
Every ordered alphabet $T=\{a,b,c,d\}$, $a<b<c<d$, is positively affinely
equivalent to
\[
 T_{x,y}=\{0,1,x,y\},\qquad
 x=\frac{c-a}{b-a},\quad y=\frac{d-a}{b-a},\qquad
 D=\{(x,y):1<x<y\}.
\]
The family $\HH$ becomes the line $H:1+x-y=0$ in $D$.
Reflection gives the only other normalized representative of the same class:
\[
 R(x,y)=\left(\frac{y-1}{y-x},\frac{y}{y-x}\right).
\]
\Needspace{12\baselineskip}
The seven isolated classes therefore give the following fourteen-point set $E$.
\begin{center}
\begin{tabular}{@{}ll@{}}\toprule
Class&Normalized representatives\\\midrule
$\{0,1,2,4\}$&$(2,4),(3/2,2)$\\
$\{0,1,2,5\}$&$(2,5),(4/3,5/3)$\\
$\{0,1,3,5\}$&$(3,5),(2,5/2)$\\
$\{0,1,3,6\}$&$(3,6),(5/3,2)$\\
$\{0,1,2,6\}$&$(2,6),(5/4,3/2)$\\
$\{0,2,3,6\}$&$(3/2,3),(4/3,2)$\\
$\{0,1,4,6\}$&$(4,6),(5/2,3)$\\\bottomrule
\end{tabular}
\end{center}
All fourteen points are distinct and lie outside $H$. Enumerating
$1\le a<b<c\le6$ with $\gcd(a,b,c)=1$ and $a+b\ne c$ gives precisely these
seven classes after reflection.

Following Freedman and Brown \cite[proofs of Theorems 1 and 2]{M-FB16}, fix a word
$W$ over labels $0,1,2,3$, interpreted as $0,1,x,y$. For adjacent blocks $U,V$
of equal positive length, let
$\delta_i=|U|_i-|V|_i$, where $|U|_i$ counts occurrences of label $i$.
Then $\sum_i\delta_i=0$, and equality of the block sums is exactly
\begin{equation}
 \delta_1+\delta_2x+\delta_3y=0.\label{M-eq:line}
\end{equation}
If the count-difference vector is zero, the pair is an abelian square and
$W$ fails for every assignment; such a candidate word is rejected.
Otherwise \eqref{M-eq:line} is a nonzero affine equation. A nonzero constant
has no zero, and only genuine lines intersecting $D$ contribute to the failure
set. Our use of this established reduction is to intersect the exact failure
sets of many words over the full two-dimensional domain.

\begin{lemma}\label{M-lem:domain}\label{M-lem:sign}
The zero set of a nonzero equation \eqref{M-eq:line} meets $D$ if and only if
\begin{equation}
 (\delta_1+\delta_2+\delta_3,\ \delta_2+\delta_3,\ \delta_3)
 \label{M-eq:gaps}
\end{equation}
has both a positive and a negative entry.
\end{lemma}
\begin{proof}
For positive consecutive gaps $p,q,r$, the four letters are
$0,p,p+q,p+q+r$. The block-sum difference is the dot product of
\eqref{M-eq:gaps} with $(p,q,r)$. A nonzero linear form vanishes at a strictly
positive vector exactly when its coefficients have both signs: balance the
positive and negative contributions, and assign arbitrary positive values to
coordinates with zero coefficient. Dividing by $p$ gives a point of $D$.
\end{proof}

For an abelian-square-free label word, let $\mathcal L(W)$ be the set of lines
passing Lemma~\ref{M-lem:domain}. Normalize each integer coefficient triple by
its gcd and by making its first nonzero entry positive. The exact failure set is
\begin{equation}
 B(W)=\bigcup_{L\in\mathcal L(W)}(L\cap D).\label{M-eq:failure}
\end{equation}
There are $\lfloor139^2/4\rfloor=4830$ adjacent equal-length block comparisons
in a word of length 139. Thus \eqref{M-eq:failure} is a finite exact description,
not merely a list of sufficient obstructions.

\section{A cover of the remaining real parameters}
\begin{proposition}\label{M-prop:cover}
The 77 words $W_1,\ldots,W_{77}$ of length 139 in Supplement S2 satisfy
\begin{equation}
 \bigcap_{i=1}^{77}B(W_i)=(H\cap D)\cup E.\label{M-eq:cover}
\end{equation}
\end{proposition}
\begin{proof}
For every listed word, compute all block-count differences, reject a zero
vector, and form the line set by Lemma~\ref{M-lem:domain}. The check verifies
that every word fails on $H$ and at all fourteen points of $E$.
For the converse, work in $D\setminus(H\cup E)$ and represent the current
intersection as whole lines together with isolated points. Start with the
first word's line union. When intersecting with the next word's line union,
retain common lines, retain existing points on at least one new line, and
add intersections of old and new distinct lines. Keep only points in the
working domain, and remove points already represented by retained lines.

For $a+bx+cy=0$ and $d+ex+fy=0$, a unique intersection exists when $bf-ce\ne0$,
and is
\begin{equation}
 x=\frac{cd-af}{bf-ce},\qquad y=\frac{ae-bd}{bf-ce}.\label{M-eq:intersection}
\end{equation}
Distinct parallel lines contribute nothing; coincident lines are retained.
These rules preserve the exact represented set at each step and require only
integer and rational arithmetic.

The checker for the word list performs all
$77\cdot4830=371{,}910$ comparisons and the stated set operations.
After the tenth word no whole residual line remains, and 125 isolated points
remain. The remaining 67 words remove all of them. Supplement S2 prints the
complete words and the intermediate line and point counts. Together with the
checks on $H\cup E$, this proves \eqref{M-eq:cover}.
\end{proof}

\begin{remark}
Irrational points can belong to an entire retained line. Only intersections
of distinct nonparallel rational lines reduce to individual rational points.
\end{remark}

\begin{remark}[The cutoff]\label{M-rem:cutoff}
With $H\cup E$ as the fixed exception set, 139 is the largest length for which
a cover is possible: $\{0,2,3,7\}$ lies outside that set and has exact value
139. A cover at length 140 would therefore need at least this additional
exception.
\end{remark}

\section{Proof of the classification}
\begin{proof}[Proof of Theorem~\ref{M-thm:main}]
Normalize $T$ to a point in $D$. On $H$, apply the published classification
\cite{M-FB16}. At a point of $E$, use Proposition~\ref{M-prop:layers}
and affine invariance.
At every other point, Proposition~\ref{M-prop:cover} gives a word of length
139 avoiding additive squares. These cases exhaust the real parameter domain.
Every listed finite value has a representative, so the spectrum below 139 is
\[
 \{50,55,58,60,62,86,88,110,111,113,118\}.
\]
The complete search at $\{0,2,3,7\}$ gives value 139 outside the exception
set and establishes sharpness.
\end{proof}

\section{The family $\{0,1,x,x+2\}$}
The endpoint-sum family has equal outer gaps. We now fix instead
\[
 T_x=\{0,1,x,x+2\},\qquad x>1.
\]
In an original alphabet $a<b<c<d$, this means $d-c=2(b-a)$.
Reflection gives the equivalent formulation with outer-gap ratio one half.

\Needspace{9\baselineskip}
\begin{theorem}\label{M-thm:line}\label{M-thm:gap-two}
For $x>1$,
\[
\begin{array}{c|ccccc}
x&2&3&4&3/2&\text{all other }x>1\\\hline
g(T_x)&62&88&118&139&\ge167.
\end{array}
\]
The last bound is sharp, since $g(T_5)=g(\{0,1,5,7\})=167$.
Thus the finite spectrum below 167 on this line is exactly
$\{62,88,118,139\}$.
\end{theorem}
\begin{proof}
Use the four length-167 words $V_0,V_1,V_2,V_3$ in Supplement S3.
For adjacent equal-length blocks, the sum difference is
\[
 (\delta_1+2\delta_3)+(\delta_2+\delta_3)x=A+Bx.
\]
All 6972 comparisons in each word have $(A,B)\ne(0,0)$.
A nonzero constant has no root; otherwise the unique failure parameter is
$x=-A/B$, retained only when $x>1$.
The four exact root sets, printed with the words in Supplement S3, have
intersection $\{3/2,2,3,4\}$. Hence every other real $x>1$, including every
irrational one, admits at least one of the four avoiding words.
This check uses all 27,888 comparisons and exact rational root lists.

The four exceptions are affinely equivalent, respectively, to
$\{0,2,3,7\}$, $\{0,1,2,4\}$, $\{0,1,3,5\}$ and $\{0,1,4,6\}$.
Their exact values, and the value 167 at $x=5$, are the complete-search
results of Proposition~\ref{M-prop:layers}. They supply the exceptional upper
bounds and sharpness; the four long words supply only the complementary
lower bound.
\end{proof}

\paragraph{Supplement and data.}
All 94 word certificates are printed below. The programs and search logs
are archived at \url{https://doi.org/10.5281/zenodo.22559145}.

\paragraph{Use of generative AI}
Generative AI tools, including OpenAI Codex, assisted with mathematical
analysis, research programming, literature review and manuscript preparation.
The author takes responsibility for the mathematics, computations and text.

\clearpage
\hypertarget{supplement-start}{}
\setcounter{section}{0}\setcounter{subsection}{0}
\setcounter{equation}{0}
\setcounter{table}{0}\setcounter{figure}{0}
\renewcommand{\MainPrefix}{Main }\renewcommand{\SuppPrefix}{}
\section*{Technical supplement}

This supplement supplies the finite words used in the main article and records
the computations on which its exact finite values depend. Direct word checks
prove lower bounds. Exact intersections of failure sets prove coverage of the
real parameter domains. Upper bounds on the maximum word length require the
separate complete searches described in Section S5. The accompanying electronic
archive contains those programs and execution records.

\renewcommand{\thesection}{S\arabic{section}}

\paragraph{How the evidence fits together.}
The plane classification uses three inputs: the published endpoint-sum
classification, complete finite searches for the seven isolated classes, and
the exact 77-word cover of their complement. Sharpness of the cutoff also
uses the complete search giving 139. The line classification uses its four-word
cover and the exact values $62,88,118,139,167$. The word-cover checks do not
replace these complete-search upper bounds. The four-letter integer case beyond the classified range remains open.

\section{Words attaining the listed finite values}\label{S-app:witnesses}
Each word below is written using the actual integers in its alphabet. Read its
rows consecutively. The displayed word has length $g$, contains no additive
square, and cannot be extended by a single letter without creating one.
Its existence proves $g(T)\ge g$; right-maximality does not rule out a different,
longer word. The matching upper bound is supplied by a complete search.
\wordentry{$T=\{0,1,2,4\}$, $g=62$}{4024142042120121420214240424142012420102014101242102\\4140421020}
\wordentry{$T=\{0,1,2,5\}$, $g=86$}{2025152025201020515251021251020105010215121025105152\\0215250525152010510121015052105020}
\wordentry{$T=\{0,1,3,5\}$, $g=88$}{5135051531013153505131510503050103101503050153151035\\315350535150103051530510130105030501}
\wordentry{$T=\{0,1,3,6\}$, $g=110$}{0106010360610131601636010363101601030106036061030106\\1636106036061636160131606306013160103010603631606316\\106163}
\wordentry{$T=\{0,1,2,6\}$, $g=111$}{0126010601262060102010610216260262120121606206120626\\1626061060201261606212021601201060102062162602010602\\6160102}
\wordentry{$T=\{0,2,3,6\}$, $g=113$}{0326320362602623063626360620603020623603623202602030\\2060363206030262026360623036062636260262306323060230\\626360636}
\wordentry{$T=\{0,1,4,6\}$, $g=118$}{0401046160641406040104064160464146404104060104101614\\1041460401416460401464064616460401606461601041460401\\06101404601041}
\wordentry{$T=\{0,2,3,7\}$, $g=139$}{7303732370727323020702032027232030702737270230373237\\3073727072732303732737027232030702707327202320270203\\02070372702723203073727073723020727}
\wordentry{$T=\{0,1,3,7\}$, $g=141$}{1737071070307010317107037071737170103137313013171013\\0107010307313707107310313710171370317370737170171317\\1070371013010703707137307017073070171}
\wordentry{$T=\{0,2,4,7\}$, $g=145$}{2024720727472702042027020424742040702742702474074727\\4707207047274074240207274207020402072420424720724704\\74270274720704070204202740474270724720704}
\wordentry{$T=\{0,1,2,7\}$, $g=148$}{1720107012172120121727072712702102071720701070201202\\7121020701020727172707121021272072717270710702721270\\20107010217012027217270717072010272101210701}
\wordentry{$T=\{0,1,4,7\}$, $g=161$}{0701040171474010704717470471017141701040107047414071\\0407471747071417101714707417407010407017141710140107\\0104147010704071074714740107101417107174717074010710\\14074}
\wordentry{$T=\{0,1,5,7\}$, $g=167$}{5715101757071015101701575051017051757075717017157057\\1510171015010701050750510170105707170757101505707170\\5751757071707507101517101507050105707170757010515710\\57050105715}

\section{Seventy-seven words for the real parameter plane}\label{S-app:cover}
The word $W_j$ uses labels $0,1,2,3$, interpreted as $0,1,x,y$. It has length
139. Its two printed rows are consecutive parts of the same word; neither the
label $W_j$ nor a line break is part of it. The symbols $H$ and $E$ denote the
exceptional line and the fourteen exceptional points defined in the main article.

All words fail on $H\cup E$. Outside that set, the successive exact intersections
have the following numbers of retained whole lines and isolated points:
\begin{quote}\small
(37,0), (19,92), (11,136), (8,140), (6,125), (4,131), (3,126), (2,123), (1,123), (0,125),
(0,114), (0,106), (0,100), (0,95), (0,90), (0,86), (0,82), (0,78), (0,74), (0,71), (0,68), (0,65),
(0,62), (0,60), (0,58), (0,56), (0,54), (0,52), (0,50), (0,48), (0,46), (0,45), (0,44), (0,43),
(0,42), (0,41), (0,40), (0,39), (0,38), (0,37), (0,36), (0,35), (0,34), (0,33), (0,32), (0,31),
(0,30), (0,29), (0,28), (0,27), (0,26), (0,25), (0,24), (0,23), (0,22), (0,21), (0,20), (0,19),
(0,18), (0,17), (0,16), (0,15), (0,14), (0,13), (0,12), (0,11), (0,10), (0,9), (0,8), (0,7), (0,6),
(0,5), (0,4), (0,3), (0,2), (0,1), (0,0).
\end{quote}
No word has a zero adjacent-block count-difference vector. Thus none contains an
abelian square, which would make it fail for every numerical assignment.
\planeentry{1}{1320103012132120121323032312302102031320301030201202312102030102032313}{230312102123203231323031030232123020103010213012023213230313032010232}
\planeentry{2}{3032313202101323031020103102320212032313230310121302313010201030212023}{202103012013032313032021202321323032312103230313023010210132302321202}
\planeentry{3}{0301020131232010302313230231013121301020103023212031020323132303121310}{131230321320301020301312131012010301021230103020310323123201031012131}
\planeentry{4}{3231323032312103231302030213121032313031321012131201031021323032131012}{010301020323132032312103123203010213121012313031323102132303231320213}
\planeentry{5}{1310302101301023212030130323123202120230130231012010301020320210301020}{130323101312302132303231302123201030231210132313010203010213120232013}
\planeentry{6}{0102010301021230102031303120212301020302012321202132302312010301202321}{323032313010201301021230210130323120212321020302010231201030102013120}
\planeentry{7}{3202321230313212010301021013121020301323130120232123203231303132120232}{132301312102030130321310121013010201030231301312102032313032312010313}
\planeentry{8}{0131012103020121323032312320212313032010203213101210130102030201213101}{201032021231303203012023212320323132303102032021231301020301021323020}
\planeentry{9}{2023031232032310121303203010210131213103203102321323032313013123032313}{203010232021031201030102031210132303102123201020323130323120213101320}
\planeentry{10}{0102010301021013101213202101301020302012031231302031021232023213230310}{302010203231210203010210131012132023130201213032030102010312102032021}
\planeentry{11}{2013230321201213212010301021031231302013032303132131021231303231301020}{103021231210130102010302303101210313230323120321310120103020301323101}
\planeentry{12}{1012310313231301021013010212321020301321301232032313230310302313203212}{010313210301020103121021231031230232130132310302030102101320232130312}
\planeentry{13}{2321023020102032312320323121012313032012032101301020103023031020130121}{323020123032313230212312102302103101203201023031303231232013101203021}
\planeentry{14}{0102010301021230102031303212320102013101202321232032310210130323132302}{321013230232123202102031320210203020103132013212010323123202123130201}
\planeentry{15}{3231323032312103231302030121013231320232131012101301023123203010201031}{012320103101210131231302013123130313230201320121323010210131210203132}
\planeentry{16}{2312101323031012101301232021013021323032313013123023121013101201030102}{032021013010230313032310120230313023213230313032031012131012030201023}
\planeentry{17}{3231323032312320232120131232030103020131210121323032312320121323020121}{320210203010210131231303202101303202123212013123130323123201213230201}
\planeentry{18}{0102010301021230213101202321232032313012010301020320210313210203202123}{212013203231323031020313212023020102030130321203231323031030212303121}
\planeentry{19}{3231323032312103120232131012101301020321323032313013123020123130131210}{121320130102010302313020121310313231303203012130102010302303121030212}
\planeentry{20}{0102010301021013101213202101301020302012023101312131032301020321301312}{303231320301303231202123102010301021230210131213103132303101232021030}
\planeentry{21}{0102010301021013101213202101301020321203013032303132101213120103010201}{321021323020102030131230323132030131232021312010313230130201312023202}
\planeentry{22}{0102010301021013101213230312101232102030201031012030201023031323130102}{132102313230323121023202103102130320301020103101230232123020103121012}
\planeentry{23}{3231323032312320232120131232032313031321310232021202301032312103102013}{032030102010310123021203010212313031232023212301020131210203010302321}
\planeentry{24}{0102010301021013101213202101303231232021013212023213230323130120320213}{120103020121323032313013121310120321201030102101312131031323031030210}
\planeentry{25}{0102010301021013101213202101301020302012023101312131032301020323121031}{323032312301202320121310302010203202132303103023132303213101201030201}
\planeentry{26}{0102010301021013101213202101301020321203013121301021032132302120230213}{103132302321232030123202131210302010203212032313230310120130320121320}
\planeentry{27}{3231323032312320232120131232032313012130320212032312301201031213103120}{230201031012101303210131202123031323130121301020103023213203013212013}
\planeentry{28}{3231323032312320232120131232032313031321310120321201030231230321310312}{302010301021013121310323010212301303230313213102123203231012310302012}
\planeentry{29}{3231323032312320232120103021232101231303132302321303132310302010203231}{201231020103010212310131230231203013032313230232103101210313230212321}
\planeentry{30}{3231323032312320232120131232030102101312320121310120103010203213013120}{213230313212010301020320212023213012132303231232021202302010302303123}
\planeentry{31}{3231323032312320232120131232032313031321302102031302312101310120103023}{031323130102123130323123202321201310203132120301303231323021231301312}
\planeentry{32}{1323031030203010213103023031323130102123212012131012010301020321230310}{321021231013123021323032313013121310302310120103023031232030130320301}
\planeentry{33}{0102010301021013101213202101303230313202123212010301021013212010313212}{013123130323123202102030131232030131210121320210203010210132120103132}
\planeentry{34}{0102010301021013101213202101301020302012023101312131032301020321232023}{101213230323123202120230120323121012313010201030212021303132023212320}
\planeentry{35}{0102010301021013101213202101301020302012023101312131032301020321232023}{101213203231323031020131013231320210203020103121020323120213101210130}
\planeentry{36}{0102010301021013101213202101301020302012023121020301312101213202313230}{102013123130323123202321201030121013230232130312101302313010201301021}
\planeentry{37}{0102010301021013101213202132302012023121013010201030230310201031323130}{131210121323032313201213120103023130201030102131210123130321232101301}
\planeentry{38}{0102010301021013101213202101301020302012023213012132303102103012023021}{031323032312320212023010323121032030103020120231210130102321023031321}
\planeentry{39}{0102010301021013101213202321230131213101203020123031030201020320213012}{130321203013032303132131032021230132130102030201202321323031021310302}
\planeentry{40}{0102010301021230102031303121021232021013010203020120232021231210301021}{232032313010210131230201202321232032313031321020302123201312310201030}
\planeentry{41}{0102010301021013101213202101301020302012023101312131032301020321232023}{101312131032312320323101201303230313020103010212301030121310131232021}
\planeentry{42}{3231323032312103231302030212312101312320323130313213101312102123032312}{101301020323123202103132131012101301020302012313031210132021023132303}
\planeentry{43}{0102010301021013101213202101301020302012132023132030102101323013121021}{231302321202321301210301202321030102013032313230212030201310120230102}
\planeentry{44}{0102010301021013101213202101301020302012031231301320120232021231013123}{201203012103132303231232012132302012131012030120232021203132302120230}
\planeentry{45}{3231323032312320232120131012103202120232130313210302303132313013120321}{203012130320301030201202301312103201203231303132131012010302312023031}
\planeentry{46}{3231323032312320232120131232032313031321310212313032021232120131020103}{231320210203010210131012132303212320103101203021232031020323132032312}
\planeentry{47}{3231323032312320232120131232032313031321310232021202301032313012101310}{232120130102010302313202320102013123130313230212313010213120232123203}
\planeentry{48}{0102010301021013101213202101301020302012023101312131032301021230231320}{301303231323023210312130323121020302101310121032313202123130323031012}
\planeentry{49}{3231323032312320232120131232032313031321302102032013213101312102320210}{132130321230201030102101321201031321202321303213101312130201031213103}
\planeentry{50}{0102010301021013101213202101312320132030103020120231210130321203231303}{132130102101321202321323031321023213101323130102013101232102032013231}
\planeentry{51}{2010302303130323120230310302010203231210121321202321323032313012103023}{012312102320210312010301020320212023031023213230310302101303203013032}
\planeentry{52}{0102010301021013101213202101301023031323130102123203231323031030213230}{321310120103132030102010310121013123130102013212021232032313230201020}
\planeentry{53}{3231323032312320232120131232032313031321310232021202301032313010212302}{010301021032131013212023031323130131201030230310201030120232132303132}
\planeentry{54}{0102010301021013101213202101301020321230231210121310232131013123130232}{120103010210313213101210130102010302303101201031012321202123103132010}
\planeentry{55}{0102010301021013101213202101301020321230231210121310323123202123103123}{010302120230312302010301021013101213230102013212013032312320212310120}
\planeentry{56}{3231323032312320232120131232032310302010203231210130102010302303120103}{012023213230201303231323023212320210203231320121312101301020103132313}
\planeentry{57}{3231323032312320232120131232030102010313210130102010302303102013012132}{302103202120232132303231210213230323130201031213103020102032021310213}
\planeentry{58}{0102010301021013101213202101301020302012031321201230130323031321013032}{030102101310120232101301020103023031012013230313032010310121013032010}
\planeentry{59}{3231323032312320232120131201031321310212320323132303103023132302010203}{202123212010301020132120213230310203132303231202123210203012101232032}
\planeentry{60}{0102010301021013101213202101301020302012031231302030131210121320213230}{310121321202123203231303201213121012320210313230102030201213121013020}
\planeentry{61}{0102010301021013101213202101301020321323010212310313231301023031030212}{320301020103101231303202120230201030102132303231320212310131230231230}
\planeentry{62}{0102010301021013101213202101301023031303230120301213120231323032312012}{321020302010231301020103101213103020102032120323123201020131013230131}
\planeentry{63}{0102010301021013101213202101303231323020123203231323031030231320321201}{031230131213101201030102123120103010203132302120230313231301312023120}
\planeentry{64}{0102010301021013101213202101301020321323020120232123203213103020310120}{130230132021020302010310120302010230313032310213102010323031020103012}
\planeentry{65}{3231323032312320232120131232032313012010323121023020102032310302303121}{013032313230232102030131213103132303231201030102013121023202103102103}
\planeentry{66}{0102010301021013121020301030230131210203012321202123101201030203013212}{023132303231201213121020310120103230310201301213230102123203231030230}
\planeentry{67}{1021232010302321232032101312320312010301020320210310212320232132303231}{301312320323103020301023213103020310120103020301302321202321303132310}
\planeentry{68}{3231323032312320232120131232032310302030103213032120213102010301021321}{012313031323102032313230232120230313231301213010210132313202320103202}
\planeentry{69}{3231323032312320232120131232032313031321302012132103203010302012320301}{303231232023213101232032313230310302321320103020301323023212320301323}
\planeentry{70}{3231323032312320232120131232032313031321201310201303231232010320212312}{102031012131012032123032131012303231320301020103121303132023213103231}
\planeentry{71}{3231323032312320232120131232032313031321310232021202301032313012101310}{232021202301021013010232132030103020313230323121032303212023202101312}
\planeentry{72}{3231323032312320232120131232032313012103102123212023010210131210230210}{323031213103021031323032312320232120103231320121320301021013121023213}
\planeentry{73}{3231323032312320232120131232021013201303230313213102123203012130102030}{201203231232012131012010302012310120232010203231320232101231301320102}
\planeentry{74}{3231323032312320232120131232032313012103102123212023101202320210203101}{213230323123020120232123203231323031030232132302321012131210230201323}
\planeentry{75}{3231323032312320232120131232032313012010313213101210130120230313023213}{203103201312313031323023213031323103020301023120231323010302313230321}
\planeentry{76}{3231323032312320232120131232032313031321302102031303202123212013120103}{023212012131210130102030132120212321013123020103231303132120212320313}
\planeentry{77}{3231323032312320212313032303103202123130321012131210232132303130320121}{310201030102132120212313023213230103023132032120103231210130102303103}

\section{Four words for the line $\{0,1,x,x+2\}$}\label{S-app:line-words}
Here $x>1$, and labels $0,1,2,3$ are interpreted as $0,1,x,x+2$.
The four words are named $V_0,V_1,V_2,V_3$, consistently with the main article.
Each has length 167; read its rows consecutively.
\wordentry{$V_0$}{2312101323031012101301232021013021323032313013123023\\1210131012010301020320210130102303130323101202303130\\2321323031303203101213101203020102303130323010212310\\23020102312}
\wordentry{$V_1$}{0102010301021013101213202101303230313202123212010301\\0210132120103132120131231303231232021020301312320301\\3121012132021020301021013212010313210301032312320310\\32021020301}
\wordentry{$V_2$}{0102010301021013101213202101303231323020123203231323\\0310302313203212010312301312131012010301021231201030\\1020313230212023031323130131202312010301202320121323\\03101201301}
\wordentry{$V_3$}{0102010301021013101213202101303231232021013212023213\\2303231301203202131201030201213230323130131213101203\\2120103010210131213103132303103021013032312102321323\\02321310313}

For a label word $V$, define
\[
 X(V)=\{x>1:V\text{ contains an additive square on }\{0,1,x,x+2\}\}.
\]
The complete failure sets are
\begin{align*}
X(V_0)&=\{4/3,3/2,2,5/2,3,4,6\},\\
X(V_1)&=\{3/2,5/3,2,5/2,3,7/2,4,5,6,7\},\\
X(V_2)&=\{4/3,3/2,2,3,4,5,6,7,8\},\\
X(V_3)&=\{8/7,7/6,6/5,5/4,9/7,4/3,7/5,3/2,8/5,5/3,\\
&\qquad7/4,9/5,11/6,2,9/4,7/3,5/2,8/3,11/4,3,10/3,\\
&\qquad7/2,11/3,4,9/2,5,11/2,7,8,9,10,11,12,13\}.
\end{align*}
Their intersection is exactly $\{3/2,2,3,4\}$.

\section{Encoding and exact coverage checks}
The words in S1 use actual integer values. Those in S2 and S3 use labels,
with the numerical assignments specified separately. The electronic data store
unbroken strings; spaces, numbering and line breaks in the printed lists are
not symbols.

For each pair of adjacent equal-length blocks, the plane checker computes the
literal count differences and the equation $a+bx+cy=0$. A zero coefficient
vector causes the candidate word to be rejected, because it is invalid for
every alphabet. For a nonzero vector, coefficients are divided by their gcd
and the first nonzero coefficient is made positive. A line meets $1<x<y$
exactly when $(a+b+c,b+c,c)$ contains both signs. Whole lines and isolated
intersection points are then intersected in exact rational arithmetic, following
the induction in the main article. The original exceptions are checked separately.
This uses $77\cdot4830=371{,}910$ block comparisons; it does not sample the
real parameter domain.

For the one-parameter family, each comparison is $A+Bx$. The checker rejects
an identically zero comparison, ignores a nonzero constant, and retains the
root $-A/B$ when $B\ne0$ and the root exceeds one. The four words require
$4\cdot6972=27{,}888$ comparisons and give the complete root lists in S3.

Direct checks of the 13 integer words and their 52 one-letter extensions
establish the printed witnesses and right-maximality. These checks and the
parameter covers do not establish upper bounds on $g(T)$.

\section{Complete searches for the finite upper bounds}
For a prefix of length $n$, let $S_j$ be the sum of its first $j$ letters.
A newly created additive square after appending a letter must end at the new
last position. Its half-length $k$ satisfies
\[
 S_n-2S_{n-k}+S_{n-2k}=0,\qquad 1\le k\le\lfloor n/2\rfloor.
\]
The search tries every alphabet letter and rejects exactly the extensions
satisfying one of these equalities. Every surviving prefix is explored.
To certify $g(T)\le g$, the complete search through length $g+1$ must finish
with no surviving prefix of that length. A computation stopped by a time or
node budget does not certify this conclusion. Nor does merely stopping at a
depth where surviving prefixes remain.

For the values $62,86,88$, the completed searches have final depths $63,87,89$
and no surviving word in the final layer. Additional CPU runs for 86
and 88 exhaust all 3372 admissible depth-eight starting prefixes. The larger
searches use recursive prefix sums, forbidden next-letter values with disjoint
subtree partitions, and literal block sums. Programs and completed-run records
are in the accompanying archive.

The table counts distinct surviving prefixes, including the empty word.
For a search split across $P$ workers, let $a_d$ count the surviving
words of length $d$. Each worker visits the entire tree through depth 12,
and each surviving prefix of length 13 starts a subtree assigned to
exactly one worker. Hence the distinct count is
\[
 \sum_{j=1}^{P}\mathrm{visits}_j-(P-1)\sum_{d=0}^{12}a_d .
\]
In table order the worker counts were $13,15,15,15$; the corresponding
numbers of prefixes through depth 12 were $106295,94749,94749,106713$.
The archive contains the split program \texttt{asf2.c} and its summaries
\texttt{par\_137.json}, \texttt{par\_127.json}, \texttt{par\_147.json} and
\texttt{par\_157.json}. For each split tree the TeX source also gives the
integer counts by depth, starting with the empty word.

\begin{center}\small
\begin{tabular}{lrrl}\toprule
Alphabet & $g$ & Distinct prefixes & Search\\\midrule
$\{0,1,2,4\}$ &62&19,097,778&direct DFS\\
$\{0,1,2,5\}$ &86&299,596,819&direct DFS\\
$\{0,1,3,5\}$ &88&234,485,149&direct DFS\\
$\{0,1,3,6\}$ &110&2,338,999,809&direct DFS\\
$\{0,1,2,6\}$ &111&2,099,553,963&direct DFS\\
$\{0,2,3,6\}$ &113&1,259,604,111&direct DFS\\
$\{0,1,4,6\}$ &118&8,319,473,669&direct DFS\\
$\{0,2,3,7\}$ &139&38,588,475,291&direct DFS\\
$\{0,1,3,7\}$ &141&44,747,049,731&split DFS\\
$\{0,2,4,7\}$ &145&23,643,323,271&direct DFS\\
$\{0,1,2,7\}$ &148&33,990,934,409&split DFS\\
$\{0,1,4,7\}$ &161&57,409,576,053&split DFS\\
$\{0,1,5,7\}$ &167&248,755,911,415&split DFS\\\bottomrule
\end{tabular}\end{center}
Two depth-first programs in the archive give the same counts for all four
diameter-six alphabets. For the partitioned runs, the maximum is the
largest depth reached in any assigned subtree; all workers must finish.
The raw worker visits remain available for comparison with the search logs.

\end{document}